\documentclass{article}
\usepackage{amsmath}
\usepackage{amsfonts}
\usepackage{amssymb}
\usepackage{graphicx}%
\usepackage{amsthm}
\providecommand{\U}[1]{\protect\rule{.1in}{.1in}}
\newtheorem{theorem}{Theorem}

\newtheorem{corollary}[theorem]{Corollary}

{\theoremstyle{definition}
\newtheorem{definition}[theorem]{Definition}}

\newtheorem{lemma}[theorem]{Lemma}

{\theoremstyle{definition}
\newtheorem{remark}[theorem]{Remark}}

\begin{document}
\sloppy 
\title{Tetrads of lines spanning $\operatorname*{PG}(7,2)$}
\author{Ron Shaw, Neil Gordon, Hans Havlicek}
\date{}
\maketitle

\begin{abstract}
Our starting point is a very simple one, namely that of a set $\mathcal{L}%
_{4}$ of four mutually skew lines in $\operatorname*{PG}(7,2).$ Under the
natural action of the stabilizer group $\mathcal{G}(\mathcal{L}_{4}%
)<\operatorname*{GL}(8,2)$ the $255$ points of $\operatorname*{PG}(7,2)$ fall
into four orbits $\omega_{1},\omega_{2},\omega_{3},\omega_{4},$ of respective
lengths $12,54,108,81.$ We show that the $135$ points $\in\omega_{2}\cup
\omega_{4}$ are the internal points of a hyperbolic quadric $\mathcal{H7}_{7}$
determined by $\mathcal{L}_{4},$ and that the $81$-set $\omega_{4}$ (which is
shown to have a sextic equation) is an orbit of a normal subgroup $\mathcal{G}%
_{81}\cong(Z_{3})^{4}$ of $\mathcal{G}(\mathcal{L}_{4}).$ There are $40$
subgroups $\cong(Z_{3})^{3}$ of $\mathcal{G}_{81},$ and each such subgroup
$H<\mathcal{G}_{81}$ gives rise to a decomposition of $\omega_{4}$ into a
triplet $\{\mathcal{R}_{H},\mathcal{R}_{H}^{\prime},\mathcal{R}_{H}%
^{\prime\prime}\}$ of $27$-sets. We show in particular that the constituents
of precisely $8$ of these $40$ triplets are Segre varieties $\mathcal{S}%
_{3}(2)$ in $\operatorname*{PG}(7,2).$ This ties in with the recent finding
that each $\mathcal{S}=\mathcal{S}_{3}(2)$ in $\operatorname*{PG}%
(7,2)$ determines a distinguished $Z_{3}$ subgroup of $\operatorname*{GL}%
(8,2)$ which generates two sibling copies $\mathcal{S}^{\prime},\mathcal{S}%
^{\prime\prime}$ of $\mathcal{S}.$

\end{abstract}

\medskip



\section{Introduction}

We work for most of the time over $\mathbb{F}_{2}=\operatorname*{GF}(2),$ and
so we can then identify a projective point $\langle
x\rangle\in\operatorname*{PG}(n-1,2)$ with the nonzero vector $x\in V(n,2)$. In
fact we will be dealing with vector space dimension $n=8,$ and we will start
out from a(ny) direct sum decomposition
\begin{equation}
V_{8}=V_{a}\oplus V_{b}\oplus V_{c}\oplus V_{d} \label{Direct sum 2+2+2+2}%
\end{equation}
of $V_{8}:=V(8,2)$ into 2-dimensional spaces $V_{a},V_{b},V_{c},V_{d}.\ $For
$h\in\{a,b,c,d\}$ we will write%
\begin{equation*}
V_{h}=\{u_{h}(\emptyset),u_{h}(0),u_{h}(1),u_{h}(2)\},\quad\text{with }%
u_{h}(\emptyset)=0. \label{vectors of V_h}%
\end{equation*}
(The reason for this labelling of the four element of $V_{h}$ is that in a
later section we wish to use $0,1,2$ as the elements of the Galois field
$\mathbb{F}_{3}.)$ So $\mathbb{P}V_{8}=\operatorname*{PG}(7,2)$ is the span of
the four projective lines
\begin{equation}
L_{h}:=\mathbb{P}V_{h}=\{u_{h}(0),u_{h}(1),u_{h}(2)\},~h\in\{a,b,c,d\}.
\label{points of L_h}%
\end{equation}
It is surprising, but gratifying, that from such a simple starting point so
many interesting and intricate geometrical aspects quickly emerge, as we now describe.

\subsection{$\mathcal{H}_{7}$-tetrads of lines in $\operatorname*{PG}(7,2)$
\label{SSec H7 tetrads of lines}}

For $v_{h}\in V_{h}$ let $(v_{a},v_{b},v_{c},v_{d}):=v_{a}\oplus v_{b}\oplus
v_{c}\oplus v_{d}$ denote a general element of $V_{8}.$ Setting $U_{ijkl}%
:=(u_{a}(i),u_{b}(j),u_{c}(k),u_{d}(l))$ then the $255$ points of
$\operatorname*{PG}(7,2)$ are
\begin{equation}
\{U_{ijkl}|~i,j,k,l\in\{\emptyset,0,1,2\},~ijkl\neq\emptyset\emptyset
\emptyset\emptyset\}. \label{255 points}%
\end{equation}
First observe that the subgroup $\mathcal{G}(\mathcal{L}_{4})$ of
$\operatorname*{GL}(8,2)$ which preserves the direct sum decomposition
(\ref{Direct sum 2+2+2+2}), and hence the foregoing tetrad
\begin{equation}
\mathcal{L}_{4}:=\{L_{a},L_{b},L_{c},L_{d}\} \label{tetrad L}%
\end{equation}
of lines, has the semi-direct product structure%
\begin{equation*}
\mathcal{G}(\mathcal{L}_{4})=\mathcal{N}\rtimes\operatorname*{Sym}%
(4),\quad\text{where }\mathcal{N}:=\operatorname*{GL}(V_{a})\times
\operatorname*{GL}(V_{b})\times\operatorname*{GL}(V_{c})\times
\operatorname*{GL}(V_{d}), \label{G(L)}%
\end{equation*}
and where $\operatorname*{Sym}(4)=\operatorname*{Sym}(\{a,b,c,d\}).$ Hence
$|\mathcal{G}(\mathcal{L}_{4})|=6^{4}\times24=31,104.$

The $\mathcal{G}(\mathcal{L}_{4}\mathcal{)}$-orbits of points are easily
determined. In addition to the weight $\operatorname{wt}(p)=\operatorname{wt}%
_{\mathcal{B}}(p)$ of a point $p\in\operatorname*{PG}(7,2)$ with respect to a
basis $\mathcal{B}$ for $V_{8},$ let us also define its \emph{line-weight}
$\operatorname{lw}(p)$ as follows:
\begin{equation*}
\operatorname{lw}(U_{ijkl})=r\quad\text{whenever precisely }r\text{ of
}i,j,k,l\text{ are in }\{0,1,2\}. \label{line weight}%
\end{equation*}
Then the $255$ points of $\operatorname*{PG}(7,2)$ clearly fall into just four
$\mathcal{G}(\mathcal{L}_{4}\mathcal{)}$-orbits $\omega_{1}$, $\omega_{2}$,
$\omega_{3}$, $\omega_{4}$, where%
\begin{equation}
\omega_{r}=\{p\in\operatorname*{PG}(7,2)|~\operatorname{lw}(p)=r\}.
\label{4 orbits omega 1,2,3,4}%
\end{equation}
The lengths of these orbits are accordingly%
\begin{equation*}
|\omega_{1}|=12,\quad|\omega_{2}|=\tbinom{4}{2}\times3^{2}=54,\quad|\omega
_{3}|=\tbinom{4}{3}\times3^{3}=108,\quad|\omega_{4}|=3^{4}=81.
\label{lengths of 4 orbits}%
\end{equation*}

\medskip

Next take note that \emph{there is a unique} $\operatorname{Sp}(8,2)$%
\emph{-geometry on} $V_{8}:=V(8,2),$ given by a non-degenerate alternating
bilinear form $B,$ such that the subspaces $V_{a},V_{b},V_{c},V_{d}$ are
hyperbolic 2-dimensional spaces which are pairwise orthogonal. If
$\mathcal{B}=\{e_{i}\}_{i\in\{1,2,3,4,5,6,7,8\}}$ is any basis such that
\begin{equation}
V_{8}=V_{a}\perp V_{b}\perp V_{c}\perp V_{d}={\prec} e_{1},e_{8}{\succ}\perp{\prec}
e_{2},e_{7}{\succ}\perp{\prec} e_{3},e_{6}{\succ}\perp{\prec} e_{4},e_{5}%
{\succ}\label{V_8 = <1,8> + ...}%
\end{equation}
then the symplectic product $x\cdot y:=B(x,y)$ is determined by its values on
basis
vectors:%
\begin{align}
e_{1}\cdot e_{8}  &  =e_{2}\cdot e_{7}=e_{3}\cdot e_{6}=e_{4}\cdot e_{5}=1,\nonumber\\
e_{i}\cdot e_{j}  &  =0\quad\text{for other values of }i,j, \label{e_i.e_j}%
\end{align}
and so has the coordinate expression
\begin{equation*}
x\cdot y=(x_{1}y_{8}+x_{8}y_{1})+(x_{2}y_{7}+x_{7}y_{2})+(x_{3}y_{6}+x_{6}%
y_{3})+(x_{4}y_{5}+x_{5}y_{4}). \label{x.y}%
\end{equation*}

Perhaps less obvious is the fact that \emph{the tetrad (\ref{tetrad L}) also
determines a particular non-degenerate quadric} $\mathcal{Q}$\emph{ in
}$\operatorname*{PG}(7,2).$ For, as we now show, such a quadric $\mathcal{Q}$
is uniquely determined by the two conditions
\begin{itemize}
\item[(i)] it has equation $Q(x)=0$ such that the quadratic form $Q$
    polarizes to give the foregoing symplectic form $B:$
    $Q(x+y)+Q(x)+Q(y)=x\cdot y$;

\item[(ii)] the $12$-set of points
\begin{equation*}
\mathcal{P}(\mathcal{L}_{4}):=\omega_{1}=L_{a}\cup L_{b}\cup L_{c}\cup
L_{d}\subset\operatorname*{PG}(7,2) \label{P(L)}%
\end{equation*}
supporting the tetrad $\mathcal{L}_{4}$\emph{ }is external to $\mathcal{Q}.$
\end{itemize}

For it follows from (i) that the terms of degree 2 in $Q$ must be
$P_{2}(x)=x_{1}x_{8}+x_{2}x_{7}+x_{3}x_{6}+x_{4}x_{5},$ and then the eight
conditions $Q(e_{i})=1$ entail that the linear terms in $Q$ must be $P_{1}(x)=%
{\textstyle\sum\nolimits_{i=1}^{8}}
x_{i},$ so that
\begin{equation}
Q(x)=P_{2}(x)+P_{1}(x)=x_{1}x_{8}+x_{2}x_{7}+x_{3}x_{6}+x_{4}x_{5}+u\cdot x,
\label{Q = P2 + P1}%
\end{equation}
where $u:=\sum\nolimits_{i=1}^{8}e_{i}.$ Further $Q$ in (\ref{Q = P2 + P1}) is
seen to satisfy also the four conditions $Q(e_{i}+e_{j})=1,~ij\in
\{18,27,36,45\},$ so indeed $Q(p)=1$ for all $p\in\omega_{1}.$

\begin{theorem}
\label{Thm H_7 = 54 + 81}The quadric $\mathcal{Q}$ is a hyperbolic quadric
$\mathcal{H}_{7};$ moreover $\mathcal{H}_{7}=\omega_{2}\cup\omega_{4}.$
\end{theorem}

\begin{proof}
There exist just two kinds, $\mathcal{E}_{7}$ and $\mathcal{H}_{7},$ of
non-degenerate quadrics in $\operatorname*{PG}(7,2).$ An elliptic quadric
$\mathcal{E}_{7}$ has $119$ points and a hyperbolic quadric $\mathcal{H}_{7}$
has $135$ points; see \cite[Theorem 5.21]{HirschfieldVol1}, \cite[Section
2.2]{ShawCliffAlg}. Since $\mathcal{Q}$ is uniquely determined, its internal
points must be a union of the $\mathcal{G}(\mathcal{L}_{4}\mathcal{)}$-orbits
$\omega_{2},\omega_{3},\omega_{4},$ of respective lengths $54,108,81.$ So the
only possibility is that $\mathcal{Q}$ is a hyperbolic quadric $\mathcal{H}%
_{7}=\omega_{2}\cup\omega_{4},$ having $54+81=135$ points. (So we will term
such a tetrad $\mathcal{L}_{4}$ of lines in $\operatorname*{PG}(7,2)$ a
$\mathcal{H}_{7}$\emph{-tetrad.)}
\end{proof}

\begin{corollary}
\label{Cor G(L) < O+(8,2)}$\mathcal{G}(\mathcal{L}_{4})$ is a subgroup of the
isometry group $\mathcal{G}(Q)\cong\operatorname{O}^{+}(8,2)<\operatorname{Sp}%
(8,2)$ of the hyperbolic quadric $\mathcal{H}_{7}.$
\end{corollary}

\begin{remark}
In fact $\mathcal{G}(\mathcal{L}_{4})$ is a maximal subgroup of
$\operatorname{O}^{+}(8,2)=\operatorname{O}_{8}^{+}(2)\cdot 2$; see \cite[p.
85]{Atlas}, where it is recorded as $S_{3}\operatorname*{wr}S_{4}.$
\end{remark}

\subsection{$\mathcal{G}(\mathcal{L}_{4})$-invariant polynomials
\label{SSec Invt Polys}}

The tetrad $\mathcal{L}_{4}=\{L_{a},L_{b},L_{c},L_{d}\}$ determines the
following $\mathcal{G}(\mathcal{L}_{4})$-invariant sets of flats in
$\operatorname*{PG}(7,2):$
\begin{align*}
\text{(i) four }5\text{-flats} &  \text{:\quad}\langle L_{a},L_{b}%
,L_{c}\rangle,~\langle L_{a},L_{b},L_{d}\rangle,~\langle L_{a},L_{c}%
,L_{d}\rangle,~\langle L_{b},L_{c},L_{d}\rangle;\nonumber\\
\text{(ii) six }3\text{-flats} &  \text{:\quad}\langle L_{a},L_{b}%
\rangle,~\langle L_{a},L_{c}\rangle,~\langle L_{a},L_{d}\rangle,~\langle
L_{b},L_{c}\rangle,~\langle L_{b},L_{d}\rangle,~\langle L_{c},L_{d}%
\rangle;\nonumber\\
\text{(iii) four }1\text{-flats} &  \text{:\quad}L_{a},L_{b},L_{c}%
,L_{d}.
\end{align*}
Let (i) $F_{hkl}=0$ be the quadratic equation of the $5$-flat $\langle
L_{h},L_{k},L_{l}\rangle$, (ii) $F_{hk}=0$ be the quartic equation of the
$3$-flat $\langle L_{h},L_{k}\rangle$ and (iii) $F_{h}=0$ be the sextic
equation of the line $L_{h}.$ (See \cite[Lemma 2]{AspectsSegreinDCC}.)
Consequently the tetrad $\mathcal{L}_{4}$ determines the $\mathcal{G}%
(\mathcal{L}_{4})$-invariant polynomials $Q_{2},Q_{4},Q_{6},$ of respective
degrees $2,4,6,$ defined as follows:%
\begin{align*}
\text{(i)~ }Q_{2} &  =F_{abc}+F_{abd}+F_{acd}+F_{bcd},\nonumber\\
\text{(ii)~ }Q_{4} &  =F_{ab}+F_{ac}+F_{ad}+F_{bc}+F_{bd}+F_{cd},\nonumber\\
\text{(iii)~ }Q_{6} &  =F_{a}+F_{b}+F_{c}+F_{d}.\label{Q_2,4,6}%
\end{align*}

\begin{theorem}
\label{Thm sextic for omega_4}The $81$-set $\omega_{4}$ has the sextic
equation $Q_{\omega_{4}}(x)=0,$ where $Q_{\omega_{4}}:=Q_{6}+Q_{4}+Q_{2}.$
\end{theorem}

\begin{proof}
Setting $\psi_{Q}:=\{p\in\operatorname*{PG}(7,2)|~Q(p)=0\},$ the last entry in
the following table follows from the three preceding entries.%
\begin{equation*}%
\begin{tabular}
[c]{cc|ccccc|cc}\cline{3-7}
&  & \multicolumn{5}{|c|}{$Q(p)$ if $p\in$} &  & \\\cline{1-2}\cline{8-9}%
\multicolumn{1}{|c}{$Q$} & \multicolumn{1}{|c|}{$\deg Q$} & $\omega_{1}$ &
$\omega_{2}$ & $\omega_{3}$ & $\omega_{4}$ &  & $\psi_{Q}$ &
\multicolumn{1}{|c|}{$|\psi_{Q}|$}\\\hline
\multicolumn{1}{|c}{$Q_{2}$} & \multicolumn{1}{|c|}{2} & 1 & 0 & 1 & 0 &  &
$\omega_{2}\cup\omega_{4}$ & \multicolumn{1}{|c|}{135}\\
\multicolumn{1}{|c}{$Q_{4}$} & \multicolumn{1}{|c|}{4} & 1 & 1 & 0 & 0 &  &
$\omega_{3}\cup\omega_{4}$ & \multicolumn{1}{|c|}{189}\\
\multicolumn{1}{|c}{$Q_{6}$} & \multicolumn{1}{|c|}{6} & 1 & 0 & 0 & 0 &  &
$\omega_{2}\cup\omega_{3}\cup\omega_{4}$ & \multicolumn{1}{|c|}{243}\\
\multicolumn{1}{|c}{$Q_{2}+Q_{4}+Q_{6}$} & \multicolumn{1}{|c|}{6} & 1 & 1 &
1 & 0 &  & $\omega_{4}$ & \multicolumn{1}{|c|}{81}\\\hline
\end{tabular}
\ \ \ \ \ \ \ \ \ \ \qedhere \label{TablePolyomega4}
\end{equation*}
\end{proof}

\begin{remark}
Of course $Q_{2}=0$ is, see Eq.~(\ref{Q = P2 + P1}), the $\mathcal{H}_{7}$
quadric $\mathcal{Q}$ of Theorem \ref{Thm H_7 = 54 + 81}. Also $Q_{4}$ was
denoted $Q_{4}^{\prime}$ in \cite[Theorem 17]{AspectsSegreinDCC} and $Q_{6}$
was denoted $Q_{6}^{\prime}$ in \cite[Example 20]{AspectsSegreinDCC}. The
sextic terms in $Q_{6}$ are readily found, since
\begin{equation*}
Q_{6}=\Pi_{i\neq1,8}(1+x_{i})+\Pi_{i\neq2,7}(1+x_{i})+\Pi_{i\neq3,6}%
(1+x_{i})+\Pi_{i\neq4,5}(1+x_{i}).\label{Q_6 =}%
\end{equation*}
Consequently, in terms of the sextic monomials $\widehat{x_{j}x_{k}}%
:=\Pi_{i\notin\{j,k\}}x_{i},$ we see that
\begin{equation}
Q_{6}=\widehat{x_{1}x_{8}}+\widehat{x_{2}x_{7}}+\widehat{x_{3}x_{6}}%
+\widehat{x_{4}x_{5}}+~(\text{terms of degree }%
<6).\label{Q_6 using ^monomials}%
\end{equation}

\end{remark}

\begin{remark}
A sextic polynomial $Q$ determines, via complete polarization, an alternating
multilinear form $\ \times^{6}V_{8}\rightarrow\mathbb{F}_{2},$ and hence an
element $b\in\wedge^{6}V_{8}^{\ast}\cong\wedge^{2}V_{8}.$ (See \cite[Section
1.1]{ShawPsiAssociate}.) Since $Q_{6}$ is $\mathcal{G}(\mathcal{L}_{4}%
)$-invariant, and since there is a unique nonzero $\mathcal{G}(\mathcal{L}%
_{4})$-invariant element of $\wedge^{2}V_{8}^{\ast}\cong\operatorname{Alt}%
(\times^{2}V_{8},\mathbb{F}_{2}),$ namely $B$ in Eq. (\ref{e_i.e_j}), it
follows, in the case $Q=Q_{6},$ that $b$ must be the $\wedge^{2}V_{8}$ image
of $B\in\wedge^{2}V_{8}^{\ast},$ namely%
\begin{equation}
b=e_{1}\wedge e_{8}+e_{2}\wedge e_{7}+e_{3}\wedge e_{6}+e_{4}\wedge e_{5}.
\label{b_Q}%
\end{equation}
It follows from (\ref{b_Q}) that the sextic terms in $Q_{6}$ must be the four
monomials in (\ref{Q_6 using ^monomials}). So the result
(\ref{Q_6 using ^monomials}) could in fact have been deduced in this
alternative manner.
\end{remark}

From the foregoing it is not too difficult to find by hand the explicit
coordinate form of the sextic polynomial $Q_{\omega_{4}}.$ In fact we used
Magma, see \cite{MAGMA}, to obtain the result below. At times writing
$1^{\prime}=8,~2^{\prime}=7,~3^{\prime}=6,~4^{\prime}=5,$ let us define the
following polynomials:%
\begin{align*}
P_{1}  &  =%
{\textstyle\sum\nolimits_{1\leq i\leq8}}
x_{i},\quad P_{2}=%
{\textstyle\sum\nolimits_{1\leq i<j\leq8}}
x_{i}x_{j},\quad P_{3}=%
{\textstyle\sum\nolimits_{1\leq i<j<k\leq8}}
x_{i}x_{j}x_{k},\nonumber\\
P_{4}  &  =%
{\textstyle\sum\nolimits_{1\leq k<l\leq4}}
x_{k}x_{k^{\prime}}x_{l}x_{l^{\prime}},\nonumber\\
P_{4}^{\prime}  &  =\sum\limits_{1\leq m\leq4}x_{m}x_{m^{\prime}}%
P_{mm^{\prime}},\quad\text{where }P_{mm^{\prime}}=\sum
\limits_{\substack{k<l,~~l\neq k^{\prime}\\k,l\notin\{m,m^{\prime}\}}%
}x_{k}x_{l},\nonumber\\
P_{5}  &  =\sum\limits_{\substack{1\leq k<l\leq4\\m\notin\{k,k^{\prime
},l,l^{\prime}\}}}x_{k}x_{k^{\prime}}x_{l}x_{l^{\prime}}x_{m},\nonumber\\
P_{6}  &  =%
{\textstyle\sum\nolimits_{1\leq k<l<m\leq4}}
x_{k}x_{k^{\prime}}x_{l}x_{l^{\prime}}x_{m}x_{m^{\prime}}=\widehat{x_{1}x_{8}%
}+\widehat{x_{2}x_{7}}+\widehat{x_{3}x_{6}}+\widehat{x_{4}x_{5}}~.
\label{Polys P1 to P6}%
\end{align*}

Then, assisted by Magma, we found that
\begin{equation*}
Q_{\omega_{4}}=P_{6}+P_{5}+P_{4}+P_{4}^{\prime}+P_{3}+P_{2}+P_{1}.
\label{Sextic for omege_4}%
\end{equation*}

\section{The eight distinguished spreads $\{\mathcal{L}_{85}^{ijk}%
\}_{i,j,k\in\{1,2\}}$}

Next we show that the partial spread $\mathcal{L}_{4}$ of four lines
determines a privileged set of eight extensions to a complete spread
$\mathcal{L}_{85}$ of $85$ lines in $\operatorname*{PG}(7,2).$ To this end,
for each $h\in\{a,b,c,d\}$ let us choose that element $\zeta_{h}%
\in\operatorname*{GL}(V_{h})$ of order 3 which effects the cyclic permutation
$(u_{h}(0)u_{h}(1)u_{h}(2))$ of the points of $L_{h}.$ Consider the eight
$Z_{3}$-subgroups $\{Z_{ijk}\}_{i,j,k\in\{1,2\}}$ of $\mathcal{G}%
(\mathcal{L}_{4})$ defined by%
\begin{equation}
Z_{ijk}=\langle A_{ijk}\rangle,\quad\text{where }A_{ijk}:=(\zeta_{a}%
)^{i}\oplus(\zeta_{b})^{j}\oplus(\zeta_{c})^{k}\oplus\zeta_{d}%
~.\label{8 Z3 subgroups}%
\end{equation}
When working using the basis $\mathcal{B}$ we will make the following choices
for the four $\zeta_{h}$ in (\ref{8 Z3 subgroups}):
\begin{align}
\zeta_{a} &  :e_{1}\mapsto e_{8}\mapsto e_{1}+e_{8},\quad\zeta_{b}%
:e_{7}\mapsto e_{2}\mapsto e_{2}+e_{7},\nonumber\\
\zeta_{c} &  :e_{3}\mapsto e_{6}\mapsto e_{3}+e_{6},\quad\zeta_{d}%
:e_{5}\mapsto e_{4}\mapsto e_{4}+e_{5}.\label{zeta-a, ... ,zeta_d}%
\end{align}
We will also choose the $u_h(0)$ so that $U_{0000}$ is the unit point $u$ of
the basis $\mathcal{B}$.
Since $(A_{ijk})^{2}+A_{ijk}+I=0,$ each $Z_{ijk}$ acts fixed-point-free on
$\operatorname*{PG}(7,2)$ and gives rise to a spread $\mathcal{L}_{85}^{ijk}$
of lines in $\operatorname*{PG}(7,2),$ with a point $p\in\operatorname*{PG}%
(7,2)$ lying on the line
\begin{equation}
L^{ijk}(p):=\{p,A_{ijk}p,(A_{ijk})^{2}p\}\in\mathcal{L}_{85}^{ijk}%
.\label{p lies on line Lijk}%
\end{equation}
Note that if in (\ref{8 Z3 subgroups}) one or more of the $\zeta_{h}$ is
replaced by the identity element $I_{h}\in\operatorname*{GL}(V_{h})$ then,
although a $Z_{3}$-subgroup of $\mathcal{G}(\mathcal{L}_{4})$ which preserves
the lines is generated, it is not fixed-point-free on $\operatorname*{PG}%
(7,2).$ So there exist precisely eight extensions of $\mathcal{L}_{4}$ to a
Desarguesian spread of $85$ lines in $\operatorname*{PG}(7,2).$ Observe that
in the case where $p$ is the unit vector $u:=\sum\nolimits_{i=1}^{8}e_{i}$ of
the basis $\mathcal{B}$ then the eight lines (\ref{p lies on line Lijk}) are
distinct: for, using $ijkl$ as shorthand for $e_{i}+e_{j}+e_{k}+e_{l},$ they
are explicitly
\begin{align}
L^{111}(u) &  =\{u,1357,2468\},~~L^{122}(u)=\{u,1256,3478\},\nonumber\\
L^{212}(u) &  =\{u,5678,1234\},~~L^{221}(u)=\{u,2358,1467\},\nonumber\\
L^{222}(u) &  =\{u,2568,1347\},~~L^{211}(u)=\{u,3578,1246\},\nonumber\\
L^{121}(u) &  =\{u,1235,4678\},~~L^{112}%
(u)=\{u,1567,2348\}.\label{8 lines through u}%
\end{align}

\begin{lemma}
For $p\in\operatorname*{PG}(7,2)$ the eight lines $L^{ijk}(p)$ are distinct if
and only if $p\in\omega_{4}.$
\end{lemma}

\begin{proof}
We have in (\ref{8 lines through u}) just seen that the eight lines are
distinct for the point $u\in\omega_{4},$ and hence for all $p\in\omega_{4}.$
Consider a point $p=(0,v_{b},v_{c},v_{d})$ of line-weight 3. Since
$A_{1jk}p=A_{2jk}p,$ and so $L^{1jk}(p)=L^{2jk}(p),$ the lines $L^{ijk}(p)$
coincide in pairs. Similarly for other points $p\in\omega_{3}.$ For a point
$p\in\omega_{2}$ of line-weight 2 the analogous reasoning shows that only two
of the lines $L^{ijk}(p)$ are distinct. And of course if $p\in\omega_{1},$
that is if $p\in L_{h}$ for some $h\in\{a,b,c,d\},$ then $L^{ijk}(p)=L_{h}$
for all eight values of $ijk.$
\end{proof}

\medskip

Recall that on a $\mathcal{H}_{7}$ quadric there exist two systems of
\emph{generators, }see \cite[Section 22.4]{HirschfeldThas3}, elements of either
system being solids ($3$-flats). Consequently it follows from the next theorem
that \emph{the foregoing eight }$Z_{3}$\emph{-subgroups of
}$\mathcal{G}(\mathcal{L}_{4})$\emph{ divide naturally into two sets of size
four, namely }$\boldsymbol{Z}$\emph{ and }$\boldsymbol{Z}^{\ast}$ where%
\begin{equation}
\boldsymbol{Z}=\{Z_{111},Z_{122},Z_{212},Z_{221}\},\quad\boldsymbol{Z}^{\ast
}=\{Z_{222},Z_{211},Z_{121},Z_{112}\}. \label{Z_0 and Z_1}%
\end{equation}

\begin{theorem}
\label{Thm Pi and PI* are generators of H_7}For $p\in\omega_{4}$ let $\Pi(p)$
denote the flat spanned by the four lines $L^{ijk}(p),~ijk\in
\{111,122,212,221\},$ and let $\Pi^{\ast}(p)$ denote the flat spanned by the
four lines $L^{ijk}(p),~ijk\in\{222,211,121,112\}.$ Then $\Pi(p)$ and
$\Pi^{\ast}(p)$ are generators of $\mathcal{H}_{7}$ which moreover belong to
different systems.
\end{theorem}

\begin{proof}
The flat $\Pi(p)=\langle L^{111}(p),L^{122}(p),L^{212}(p),L^{221}(p)\rangle$
for the point $p=(v_{a},v_{b},v_{c},v_{d})\in\omega_{4}$ is seen, upon using
$(\zeta_{h})^{2}+\zeta_{h}=I_{h},$ to consist of the nine points
$L^{111}(p)\cup L^{122}(p)\cup L^{212}(p)\cup L^{221}(p)$ of line-weight 4
together with the following six points of line-weight 2:%
\begin{align}
&(v_{a},v_{b},0,0),~(v_{a},0,v_{c},0),~(v_{a},0,0,v_{d}),\nonumber\\
&(0,0,v_{c},v_{d}),~(0,v_{b},0,v_{d}),~(0,v_{b},v_{c},0).
\label{6 points of line-weight 2}%
\end{align}
By Theorem \ref{Thm H_7 = 54 + 81}, for each $p\in\omega_{4},$ the flat
$\Pi(p)$ is in fact a solid on the quadric $\mathcal{H}_{7}.$ Similarly the
same applies to the flat $\Pi^{\ast}(p),$ whose six points of line-weight 2
moreover coincide with those of $\Pi(p).$ So $\Pi(p)\cap\Pi^{\ast}(p)$ is the
isotropic plane consisting of $p=(v_{a},v_{b},v_{c},v_{d})$ together with the
six points (\ref{6 points of line-weight 2}). Consequently, see \cite[Theorem
22.4.12, Corollary]{HirschfeldThas3}, for each $p\in\omega_{4}$ the generators
$\Pi(p)$ and $\Pi^{\ast}(p)$ belong to different systems.
\end{proof}

\section{The normal subgroup $\mathcal{G}_{81}$ of $\mathcal{G}(\mathcal{L}%
_{4})$}

Let $Z_{h}$ denote that $Z_{3}$ subgroup of $\mathcal{G}(\mathcal{L}_{4})$
which fixes pointwise each of the three lines $\mathcal{L}_{4}\setminus
L_{h},~h\in\{a,b,c,d\}.$ Then clearly the elementary abelian group
\begin{equation*}
\mathcal{G}_{81}:=Z_{a}\times Z_{b}\times Z_{c}\times Z_{d}\cong Z_{3}\times
Z_{3}\times Z_{3}\times Z_{3} \label{G_81}%
\end{equation*}
is a normal subgroup of $\mathcal{N}=\operatorname*{GL}(V_{a})\times
\operatorname*{GL}(V_{b})\times\operatorname*{GL}(V_{c})\times
\operatorname*{GL}(V_{d})$ and also of $\mathcal{G}(\mathcal{L}_{4}%
)=\mathcal{N}\rtimes\operatorname*{Sym}(4).$ Observe that $\omega_{4}$ is a
single $\mathcal{G}_{81}$-orbit. One easily sees that $\mathcal{G}_{81}$ is
equally well the direct product $Z_{111}\times Z_{122}\times Z_{212}\times
Z_{221}$ of the four members of $\boldsymbol{Z,}$ and also the direct product
$Z_{222}\times Z_{211}\times Z_{121}\times Z_{112}$ of the four members of
$\boldsymbol{Z}^{\ast}.$

Consider now \emph{any} $Z_{3}\times Z_{3}\times Z_{3}$ subgroup
$H<\mathcal{G}_{81}$. If $\mathcal{G}_{81}=H\cup H^{\prime}\cup H^{\prime
\prime}$ denotes the decomposition of $\mathcal{G}_{81}$ into the cosets of
$H$ then we define subsets $\mathcal{R}:=\mathcal{R}_{H},\mathcal{R}^{\prime
}:=\mathcal{R}_{H}^{\prime},\mathcal{R}^{\prime\prime}:=\mathcal{R}%
_{H}^{\prime\prime}$ of $\omega_{4}$ by
\begin{equation}
\mathcal{R}=\{hu,h\in H\},~\mathcal{R}^{\prime}=\{h^{\prime}u,h^{\prime}\in
H^{\prime}\},~\mathcal{R}^{\prime\prime}=\{h^{\prime\prime}u,h^{\prime\prime
}\in H^{\prime\prime}\}.\label{R, R', R''}%
\end{equation}
In particular $\mathcal{R}=\mathcal{R}_{H}$ is the orbit of $u$ under the
action of the group $H.$ Each such subgroup $H<\mathcal{G}_{81}$ gives rise to
a decomposition $\omega_{4}=\mathcal{R}\cup\mathcal{R}^{\prime}\mathcal{\cup
R}^{\prime\prime}$ of $\omega_{4}$ into a triplet of $27$-sets.

As we will now demonstrate, \emph{the study of such triplets is greatly
simplified by viewing }$\mathcal{G}_{81}$\emph{ in a }$\operatorname*{GF}%
(3)$\emph{ light.}

\subsection{A $\operatorname*{GF}(3)$ view of $\mathcal{G}_{81}$
\label{SSec GF(3) view of G_81}}

For $i,j,k,l\in\mathbb{F}_{3}=\operatorname*{GF}(3)=\{0,1,2\}$ define%
\begin{equation*}
A_{ijkl}:=(\zeta_{a})^{i}\oplus(\zeta_{b})^{j}\oplus(\zeta_{c})^{k}%
\oplus(\zeta_{d})^{l}~.\label{Aijkl = zetas}%
\end{equation*}
\emph{Note that if }$i,j,k\in\{1,2\}$\emph{ then }$A_{ijk1}=A_{ijk},$\emph{ as
previously defined in (\ref{8 Z3 subgroups}).} In the following we will view
$ijkl$ as shorthand for the element $(i,j,k,l)\in(\mathbb{F}_{3})^{4}.$ Since
\begin{equation}
A_{\mathcal{\sigma}}A_{\tau}=A_{\mathcal{\sigma}+\tau},~\mathcal{\sigma}%
,\tau\in(\mathbb{F}_{3})^{4},\label{isomorphism A}%
\end{equation}
observe that $A:\sigma\mapsto A_{\sigma},~\mathcal{\sigma}\in(\mathbb{F}%
_{3})^{4},$ is an isomorphism mapping the additive group $(\mathbb{F}_{3}%
)^{4}$ onto the multiplicative group $\mathcal{G}_{81}.$ Now the orbit of any
point $p\in\omega_{4}$ under the action of the group $\mathcal{G}_{81}\ $is
the whole of $\omega_{4}.$ In particular this is so for the unit point
$u:=\sum\nolimits_{i=1}^{8}e_{i}$ of the basis $\mathcal{B}.$
\emph{Consequently the }$81$\emph{-set }$\omega_{4}$\emph{ is in bijective
correspondence with }$(\mathbb{F}_{3})^{4}$\emph{ as given by the map }%
$\theta_{u}:(\mathbb{F}_{3})^{4}\rightarrow\omega_{4}$ defined by\emph{ }%
\begin{equation}
\theta_{u}(\sigma)=p_{\sigma}:=A_{\sigma}u,~\sigma\in(\mathbb{F}_{3}%
)^{4}.\label{Def theta_u}%
\end{equation}
Observe that the choices made in (\ref{points of L_h}) and
(\ref{zeta-a, ... ,zeta_d}) imply that $\theta_{u}(ijkl)=U_{ijkl}$ for all
$ijkl\in(\mathbb{F}_{3})^{4}.$

In the $\operatorname*{GF}(3)$ space $V(4,3)=(\mathbb{F}_{3})^{4}$ we will
chiefly employ the basis $\mathcal{B}_{\varepsilon}:=\{\varepsilon
_{1},\varepsilon_{2},\varepsilon_{3},\varepsilon_{4}\},$ where%
\begin{equation}
\varepsilon_{1}=1000,~\varepsilon_{2}=0100,~\varepsilon_{3}=0010,~\varepsilon
_{4}=0001, \label{epsilon_1,2,3,4}%
\end{equation}
and then write a general element $\xi=%
{\textstyle\sum\nolimits_{r=1}^{4}}
\xi_{r}\varepsilon_{r}\in V(4,3)$ as $\xi=\xi_{1}\xi_{2}\xi_{3}\xi_{4}.$ We
denote the weight of $\xi\in(\mathbb{F}_{3})^{4}$ in the basis $\mathcal{B}%
_{\varepsilon}$ by $\operatorname{wt}_{\varepsilon}(\xi).$

\smallskip

\emph{Let us now study subgroups of }$\mathcal{G}_{81}$ \emph{by viewing them
in the light of their corresponding subspaces in the vector space
}$V(4,3)=(\mathbb{F}_{3})^{4}.$

A $Z_{3}$ subgroup of $\mathcal{G}_{81}\ $is of the form $\{I,A_{\sigma
},A_{2\sigma}\}$ for some non-zero $\sigma\in V(4,3).$ \emph{So }%
$\mathcal{G}_{81}$\emph{ contains }$40$\emph{ subgroups }$\cong Z_{3}$
\emph{which are in bijective correspondence with the }$40$\emph{ points of the
projective space }$\operatorname*{PG}(3,3)=\mathbb{P}V(4,3)$. If we denote by
$\Xi\cup\Xi^{\ast}$ the following eight elements of $(\mathbb{F}_{3})^{4}:$
\begin{align}
\Xi &  :\quad\mathcal{\alpha}=1111,~\mathcal{\beta}=1221,~\mathcal{\gamma
}=2121,~\mathcal{\delta}=2211,\nonumber\\
\Xi^{\ast} &  :\quad\mathcal{\alpha}^{\ast}=2221,~\mathcal{\beta}^{\ast
}=2111,~\mathcal{\gamma}^{\ast}=1211,~\mathcal{\delta}^{\ast}%
=1121,\label{XiXi* alpha ... delta*}%
\end{align}
then observe that $A_{\alpha},A_{\beta},\ldots,A_{\delta^{\ast}}$ are the
respective generators of the eight $Z_{3}$-subgroups
$Z_{111},Z_{122},\ldots,Z_{112}\in\boldsymbol{Z}\cup\boldsymbol{Z}^{\ast}$
considered in
(\ref{Z_0 and Z_1}). Now under the action by conjugacy of $\mathcal{G}%
(\mathcal{L}_{4})$ on $\mathcal{G}_{81}$ the particular $4$-set $\{Z_{a}%
,Z_{b},Z_{c},Z_{d}\}=\{\langle A_{\varepsilon_{1}}\rangle,\langle
A_{\varepsilon_{2}}\rangle,\langle A_{\varepsilon_{3}}\rangle,\langle
A_{\varepsilon_{4}}\rangle\}$ of $Z_{3}$ subgroups is fixed, whence
\begin{equation*}
\mathcal{T}_{\varepsilon}:=\{\langle\varepsilon_{1}\rangle,\langle
\varepsilon_{2}\rangle,\langle\varepsilon_{3}\rangle,\langle\varepsilon
_{4}\rangle\}
\end{equation*}
\emph{is a }$\mathcal{G}(\mathcal{L}_{4})$\emph{-distinguished
tetrahedron of reference in }$\operatorname*{PG}(3,3).$ Consequently take note
that \emph{the eight }$Z_{3}$\emph{ subgroups }$\{\langle A_{\rho}%
\rangle\}_{\rho\in\Xi\cup\Xi^{\ast}}$\emph{ considered in (\ref{Z_0 and Z_1})
are picked out as the only }$Z_{3}$\emph{ subgroups }$\langle A_{\rho}\rangle$
of $\mathcal{G}_{81}$\emph{ for which} $\operatorname{wt}_{\varepsilon}%
(\rho)=4.$

Next let us consider subgroups $H\cong Z_{3}\times Z_{3}\times Z_{3}$ of
$\mathcal{G}_{81}.$

\begin{theorem}
\label{Thm 40 subgps}The normal subgroup $\mathcal{G}_{81}<\mathcal{G}%
(\mathcal{L}_{4})$ contains precisely $40$ subgroups $H\cong Z_{3}\times
Z_{3}\times Z_{3}.$ These fall into four conjugacy classes $\mathcal{C}%
_{0},\mathcal{C}_{1},\mathcal{C}_{2},\mathcal{C}_{3}$ of $\mathcal{G}%
(\mathcal{L}_{4}),$ of respective sizes $8,16,12,4.$
\end{theorem}

\begin{proof}
Each subgroup $\langle A_{\rho},A_{\sigma},A_{\tau}\rangle\cong Z_{3}\times
Z_{3}\times Z_{3}$ arises as
\begin{equation*}
\{A_{\lambda}|~\lambda\in V_{3}:={{\prec} \rho,\sigma,\tau{\succ}}\}
\end{equation*}
from a corresponding projective plane $\mathbb{P}%
V_{3}=\langle\langle\rho\rangle,\langle\sigma\rangle,\langle\tau\rangle
\rangle$ in $\operatorname*{PG}(3,3).$ Now there exist precisely $40$ planes
in $\operatorname*{PG}(3,3),$ and these fall into four kinds $\mathcal{P}%
_{0}$, $\mathcal{P}_{1}$, $\mathcal{P}_{2}$, $\mathcal{P}_{3},$ where
$\mathcal{P}_{r}$ denotes those planes in $\operatorname*{PG}(3,3)$ which
contain precisely $r$ of the vertices $\langle\varepsilon_{i}\rangle$ of the
tetrahedron of reference $\mathcal{T}_{\varepsilon}.$ There are $8$ planes of
kind $\mathcal{P}_{0},$ namely those with one of the $8$ equations%
\begin{equation}
\xi_{4}=c_{1}\xi_{1}+c_{2}\xi_{2}+c_{3}\xi_{3},\quad c_{1},c_{2},c_{3}%
\in\{1,2\}.\label{8 type P_0 planes}%
\end{equation}
Similarly we see that there are, respectively, $16$, $12$, $4$ planes of kinds
$\mathcal{P}_{1}$, $\mathcal{P}_{2}$, $\mathcal{P}_{3}.$ The theorem now
follows, since planes of the same kind are seen to correspond to conjugate
$Z_{3}\times Z_{3}\times Z_{3}$ subgroups.
\end{proof}

Finally let us consider subgroups $H\cong Z_{3}\times Z_{3}$ of $\mathcal{G}%
_{81}.$ Such a subgroup $\langle A_{\rho},A_{\sigma}\rangle$ arises from a
corresponding line $\langle\langle\rho\rangle,\langle\sigma\rangle
\rangle\subset\operatorname*{PG}(3,3),$ and so we need to classify lines with
respect to the $\mathcal{G}(\mathcal{L}_{4})$-distinguished basis
$\mathcal{B}_{\varepsilon}.$ If $n_{w}$ points of a line $L\subset
\operatorname*{PG}(3,3)$ have weight $w,~w\in\{1,2,3,4\},$ with respect to the
basis $\mathcal{B}_{\varepsilon}\ $then we will say that $L$ has \emph{weight
pattern} $\pi_{\varepsilon}(L)=(n_{1},n_{2},n_{3},n_{4}).$

\begin{theorem}
\label{Thm 130 subgps}The normal subgroup $\mathcal{G}_{81}<\mathcal{G}%
(\mathcal{L}_{4})$ contains precisely $130$ subgroups $\cong Z_{3}\times
Z_{3}.$ These fall into seven conjugacy classes $\mathcal{K}_{1}%
,\ldots,\mathcal{K}_{7}$ of $\mathcal{G}(\mathcal{L}_{4}),$ of respective sizes
$6,24,16,12,16,48,8.$
\end{theorem}

\begin{proof}
Each subgroup $\langle A_{\rho},A_{\sigma}\rangle\cong Z_{3}\times Z_{3}$
arises from a corresponding line $\langle\langle\rho\rangle,\langle
\sigma\rangle\rangle\subset\operatorname*{PG}(3,3),$ and there exist precisely
$130$ lines in $\operatorname*{PG}(3,3).$ With respect to the $\mathcal{G}%
(\mathcal{L}_{4})$-distinguished tetrahedron of reference $\mathcal{T}%
_{\varepsilon}$ the $130$ lines $L$ are of seven kinds $\Lambda_{1}%
,\ldots,\Lambda_{7}$ as described in the following table.
\begin{equation}%
\begin{tabular}
[c]{ccc}\hline
& $\pi_{e}(L)$ & $|\Lambda_{i}|$\\\hline
$L\in\Lambda_{1}$ & ${ (2,2,0,0)}$ & $6$\\
$L\in\Lambda_{2}$ & ${ (1,1,2,0)}$ & $24$\\
$L\in\Lambda_{3}$ & ${ (0,3,1,0)}$ & $16$\\
$L\in\Lambda_{4}$ & ${ (0,2,0,2)}$ & $12$\\
$L\in\Lambda_{5}$ & ${ (1,0,1,2)}$ & $16$\\
$L\in\Lambda_{6}$ & ${ (0,1,2,1)}$ & $48$\\
$L\in\Lambda_{7}$ & ${ (0,0,4,0)}$ & $8$\\\hline
\end{tabular}
\label{Table 130 lines}%
\end{equation}
The theorem now follows since lines of the same kind correspond to conjugate
$Z_{3}\times
Z_{3}$ subgroups.%
\end{proof}

\subsection{The 27-set `denizens' of $\omega_{4}$}

It follows from Theorem \ref{Thm 40 subgps} that the $81$-set $\omega_{4}$ is
populated by $40$ triplets $\{\mathcal{R},\mathcal{R}^{\prime},\mathcal{R}%
^{\prime\prime}\}$ of $27$-set `denizens', as in (\ref{R, R', R''}), and that
these triplets, and the $120$ denizens of $\omega_{4}$, can be classified into
four kinds $\mathcal{C}_{0},\mathcal{C}_{1},\mathcal{C}_{2},\mathcal{C}_{3}.$
One of our aims is to show that \emph{precisely eight of these triplets are
triplets of Segre varieties $\mathcal{S}_{3}(2).$ }So it helps to remind
ourselves at this point about certain aspects of a Segre variety
$\mathcal{S}=\mathcal{S}_{3}(2)$ in $\operatorname*{PG}(7,2),$ and to relate
our present concerns to those in \cite{Glynn et al},
\cite{HavlicekOdehnalSaniga} and \cite{AspectsSegreinDCC}.

First of all, $\mathcal{S}$ defines a $(27_{3},27_{3})$ configuration, each of
the $27$ points of $\mathcal{S}$ lying on \emph{precisely }$3$ lines
$\subset\mathcal{S},$ namely three of the $27$ generators of $\mathcal{S}.$
Moreover the stabilizer group $\mathcal{G}_{\mathcal{S}}$ of $\mathcal{S}$
contains as a normal subgroup a group $\langle A_{1},A_{2},A_{3}\rangle\cong
Z_{3}\times Z_{3}\times Z_{3}$ which acts transitively on the $27$ points of
$\mathcal{S},$ the three generators of $\mathcal{S}$ through a point
$p\in\mathcal{S}$ being the lines
\begin{equation}
L^{r}(p):=\{p,A_{r}p,(A_{r})^{2}p\},~r=1,2,3.
\label{3 generators thro a point of S}%
\end{equation}
Here $A_{r}$ satisfies $(A_{r})^{2}+A_{r}+I=0,$ each $Z_{3}$ group $\langle
A_{r}\rangle$ acting fixed-point-free on $\operatorname*{PG}(7,2).$ Further, as
noted in \cite[Theorem 5]{AspectsSegreinDCC}, $\mathcal{S}$ determines a
distinguished $Z_{3}$-subgroup $\langle W\rangle$ which also acts
fixed-point-free on $\operatorname*{PG}(7,2),$ the distinguished tangent, see
\cite[Section 2.1]{AspectsSegreinDCC}, at $p\in\mathcal{S}$ being the line
$\{p,Wp,W^{2}p\}.$ Moreover, see \cite[Section 4.2]{AspectsSegreinDCC}, under
the action of the distinguished $Z_{3}$-subgroup $\langle W\rangle$ the Segre
variety $\mathcal{S}$ gave rise to a triplet $\{\mathcal{S},~\mathcal{S}%
^{\prime}=W(\mathcal{S}),~\mathcal{S}^{\prime\prime}=W^{2}(\mathcal{S})\}$ of
Segre varieties. In \cite[p.~82]{Glynn et al} (although without proof and using
a different notation), \cite[Proposition 5]{HavlicekOdehnalSaniga} and
\cite[Section 4.1]{AspectsSegreinDCC} the five
$\mathcal{G}_{\mathcal{S}}$-orbits $\mathcal{O}_{1},\mathcal{O}_{2}%
,\mathcal{O}_{3},\mathcal{O}_{4},\mathcal{O}_{5}$ of points were described,
with $\mathcal{O}_{5}=\mathcal{S}$ and $\mathcal{O}_{4}=\mathcal{S}^{\prime
}\cup\mathcal{S}^{\prime\prime}.$  These are related to the four $\mathcal{G}%
(\mathcal{L}_{4})$-orbits (\ref{4 orbits omega 1,2,3,4}) in the following
simple manner:%
\begin{equation}
\mathcal{\omega}_{1}=\mathcal{O}_{1},\quad\mathcal{\omega}_{2}=\mathcal{O}%
_{2},\quad\mathcal{\omega}_{3}=\mathcal{O}_{3},\quad\mathcal{\omega}%
_{4}=\mathcal{O}_{4}\cup\mathcal{O}_{5}=\mathcal{S}\cup\mathcal{S}^{\prime
}\cup\mathcal{S}^{\prime\prime}. \label{omega orbits and O orbits}%
\end{equation}
So $\omega_{4}$ is a single orbit under the action of the group $\langle
A_{1},A_{2},A_{3},W\rangle\cong(Z_{3})^{4},$ this last thus being the group
$\mathcal{G}_{81}$ in our present context.

\begin{lemma}\label{lines in omega4}
If $p\in\omega_{4}$ and $\lambda\in(\mathbb{F}_{3})^{4},\lambda\neq0000,$ then
$L_{p}^{\lambda}:=\{p,A_{\lambda}p,A_{2\lambda}p\}$ is a line in $\omega_{4}$
if and only if $\pm\lambda\in\Xi\cup\Xi^{\ast}.$
\end{lemma}

\begin{proof}
We already know, see (\ref{p lies on line Lijk}), that $L_{p}^{\lambda}$ is a
line if $\lambda\in\Xi\cup\Xi^{\ast}$ or if $-\lambda\in\Xi\cup\Xi^{\ast}.$
Also, as noted after equation (\ref{XiXi* alpha ... delta*}), if $\pm
\lambda\notin$ $\Xi\cup\Xi^{\ast}$ then $\operatorname{wt}_{\varepsilon
}(\lambda)<4,$ and so $\lambda=mnrs$ where at least one of $m,n,r,s$ is $0.$
For example, suppose $\lambda=mnr0,$ where $m,n,r\in\mathbb{F}_{3}.$ Then
$(I+A_{\lambda}+A_{2\lambda})U_{ijkl}=U_{\emptyset\emptyset\emptyset l}\neq0.$
\end{proof}

\begin{remark}\label{partial_affine}
A \emph{partial affine space}, see \cite[p.~35]{Blunck Herzer} or
\cite[p.~794]{Herzer}, is an affine space from which some parallel classes have
been removed. For example, the affine space on $(\mathbb{F}_{3})^{4}$ turns
into a partial affine space if we consider only affine lines with a direction
vector $\pm\lambda\in\Xi\cup\Xi^{\ast}$ and restrict the parallelism of
$(\mathbb{F}_{3})^{4}$ to the set of those lines. Lemma~\ref{lines in omega4}
shows that $\omega_{4}$ arises as the point set of an isomorphic partial affine
space in the following way: The \emph{lines} in $\omega_4$ are of the form
$L_{p}^{\lambda}$ with $\pm\lambda\in\Xi\cup\Xi^{\ast}$. Two lines are
\emph{parallel} if they belong to the same distinguished spread
$\mathcal{L}_{85}^{ijk}$.
\end{remark}

\begin{theorem}
\label{Thm Only P_0 gives Segres}A triplet of $27$-sets $\{\mathcal{R}%
_{H},\mathcal{R}_{H}^{\prime},\mathcal{R}_{H}^{\prime\prime}\}$ in
(\ref{R, R', R''}) which arises from a $(Z_{3})^{3}$ subgroup $H=\{A_{\lambda
}|~\lambda\in V_{3}\}$ will consist of Segre varieties $\mathcal{S}_{3}(2)$ if
and only if the projective plane $P=\mathbb{P}V_{3}\subset\operatorname*{PG}%
(3,3)$ is of kind $\mathcal{P}_{0}.$ So the $81$-set $\omega_{4}$ contains
precisely $24$ copies of a Segre variety $S_{3}(2).$
\end{theorem}

\begin{proof}
Since each point of a Segre $\mathcal{S}=$ $\mathcal{S}_{3}(2)$ lies on
\emph{precisely }three generators of $\mathcal{S},$ see (\ref{3 generators thro
a point of S}), it follows from the preceding lemma that in order for
$\mathcal{R}_{H}$ to be a $\mathcal{S}_{3}(2)$ the subgroup $H$ must be of the
form $\langle A_{\lambda},A_{\mu},A_{\nu}\rangle$ for \emph{precisely }three
element $\lambda,\mu,\nu\in\Xi\cup\Xi^{\ast}.$ But a
straightforward check shows that planes of the kinds $\mathcal{P}%
_{0},\mathcal{P}_{1},\mathcal{P}_{2},\mathcal{P}_{3}$ contain, respectively,
precisely $3,2,4,0$ points $\langle\lambda\rangle$ with $\lambda\in\Xi\cup
\Xi^{\ast}.$ So only in the eight cases where the $Z_{3}$ subgroup
$H<\mathcal{G}_{81}$ is of kind $\mathcal{C}_{0}$ can $\mathcal{R}_{H}$ be
Segre variety $S_{3}(2).$ By (\ref{omega orbits and O orbits}) for one such
subgroup a Segre variety arises and, by Theorem~\ref{Thm 40 subgps}, the same
holds for the remaining subgroups of kind $\mathcal{C}_{0}$.
\end{proof}

\begin{remark}
\label{Rmk all Xi or all Xi* for kind P_0}It is easy to check that if all
three of $\lambda,\mu,\nu$ are in $\Xi,$ or if all three are in $\Xi^{\ast},$
then $\mathbb{P}({\prec}\lambda,\mu,\nu{\succ})$ is a plane in $\operatorname*{PG}%
(3,3)$ of kind $\mathcal{P}_{0},$ thus accounting for all eight planes of kind
$\mathcal{P}_{0}.$ But if $\lambda,\mu,\nu$ are split $2,1$ or $1,2$ between
$\Xi,\Xi^{\ast},$ then we see that $H=\langle A_{\lambda},A_{\mu},A_{\nu
}\rangle$ contains a fourth $Z_{3}$ subgroup $\langle A_{\rho}\rangle$ with
$\rho\in\Xi\cup\Xi^{\ast}:$ for example, observe results such as
\begin{equation}
{\prec}\alpha,\beta,\alpha^{\ast}{\succ}={\prec}\alpha,\beta,\alpha^{\ast}%
,\beta^{\ast}{\succ},\quad\text{and\quad}{\prec}\alpha,\beta,\gamma^{\ast}%
{\succ}={\prec}\alpha,\beta,\gamma^{\ast},\delta^{\ast}{\succ}.
\label{a,b,a* and a,b,g*}%
\end{equation}
So such a $(Z_{3})^{3}$ subgroup $H$ is of kind $\mathcal{C}_{2},$ not
$\mathcal{C}_{0},$ and $\mathcal{R}_{H}$ gives rise to a $(27_{4},36_{3})$
configuration in contrast to the $(27_{3},27_{3})$ configuration arising from
a $S_{3}(2).$
\end{remark}

It was proved in \cite[Theorem 18]{AspectsSegreinDCC} that a Segre variety
$S_{3}(2)$ in $\operatorname*{PG}(7,2)$ has a sextic equation. In fact, from
our results in Section \ref{SSec Invt Polys} we can deduce that if
$\mathcal{S}$ is any of the $24$ copies of a Segre variety $S_{3}(2)$ in
$\omega_{4}$ then $\mathcal{S}$ has a sextic equation of the form
$Q_{\mathcal{S}}(x)=0$ where, for some polynomial $F_{\mathcal{S}}$ of degree
$<6,$%
\begin{equation}
Q_{\mathcal{S}}=\widehat{x_{1}x_{8}}+\widehat{x_{2}x_{7}}+\widehat{x_{3}x_{6}%
}+\widehat{x_{4}x_{5}}+F_{\mathcal{S}}. \label{Q_S =}%
\end{equation}
For recall from (\ref{omega orbits and O orbits}) that $\mathcal{\omega}%
_{4}=\mathcal{O}_{4}\cup\mathcal{O}_{5}$ where $\mathcal{O}_{5}=\mathcal{S}.$
Now, see \cite[Theorem 17]{AspectsSegreinDCC}, the $201$-set $(\mathcal{O}%
_{4})^{\text{c}}:=\mathcal{\omega}_{1}\cup\mathcal{\omega}_{2}\cup
\mathcal{\omega}_{3}\cup\mathcal{S}$ has a quartic equation, say
$F_{\mathcal{S}}^{\prime}=0,$ and, see Theorem~\ref{Thm sextic for omega_4}, $\mathcal{O}%
_{4}\cup\mathcal{S}$ has sextic equation $Q_{6}+Q_{4}+Q_{2}=0.$ It follows
that $\mathcal{S}$ has equation $Q_{6}+Q_{4}+Q_{2}+F_{\mathcal{S}}^{\prime}=0$
which, see (\ref{Q_6 using ^monomials}), is of the form (\ref{Q_S =}).

\subsection{Non-Segre triplets in $\omega_{4}$}

When the authors first considered $27$-sets such as $\mathcal{R}_{\alpha
\beta\gamma^{\ast}}=\{A_{\rho}u\}_{\rho\in{\prec}\alpha,\beta,\gamma^{\ast}%
{\succ}}$ they were briefly misled into thinking that $\mathcal{R}_{\alpha
\beta\gamma^{\ast}}$ was a Segre $\mathcal{S}_{3}(2).$ Upon discovering their
error, see (\ref{a,b,a* and a,b,g*}), they decided that all the `deceitful'
non-Segre $27$-sets in $\omega_{4}$ should be termed \emph{`rogues'.} In this
terminology we can summarize our foregoing results as follows.

\begin{quotation}
The $81$-set $\omega_{4}$ is populated by $120$ denizen $27$-sets which occur
as $40$ triplets $\{\mathcal{R},\mathcal{R}^{\prime},\mathcal{R}^{\prime\prime
}\}$ such that $\mathcal{R}\cup\mathcal{R}^{\prime}\cup\mathcal{R}%
^{\prime\prime}=\omega_{4}.$ Of these triplets eight are of Segre varieties
$\mathcal{S}_{3}(2),$ sixteen are of rogues of kind $\mathcal{C}_{1},$ twelve
are of rogues of kind $\mathcal{C}_{2}$ and four are of rogues of kind
$\mathcal{C}_{3}.$
\end{quotation}

Since a Segre variety $\mathcal{S}_{3}(2)$ spans $\operatorname*{PG}(7,2),$
the next two theorems confirm in a more vivid manner that rogues of kinds
$\mathcal{C}_{2}$ and $\mathcal{C}_{3}$ are not Segre varieties.

\begin{theorem}
Suppose that a $27$-set $\mathcal{R\subset\omega}_{4}$ is a rogue of kind
$\mathcal{C}_{2}.$ Then $\langle\mathcal{R}\rangle$ is a $5$-flat in
$\operatorname*{PG}(7,2).$
\end{theorem}

\begin{proof}
Six of the twelve planes of kind $\mathcal{P}_{2}$ are those having equations
of the kind $\xi_{r}=\xi_{s}$ and the other six are those having equations of
the kind $\xi_{r}=2\xi_{s}.$ Consider a plane $P=\mathbb{P}V_{3}$ in the first
six, say with equation $\xi_{3}=\xi_{4}.$ Then, see (\ref{255 points}), the
subset $\theta_{u}(V_{3})$ consists of the $27$ points
\begin{equation}
\mathcal{R}:=\{U_{ijkk}|~i,j,k\in\mathbb{F}_{3}=\{0,1,2\}\}. \label{R = Uijkk}%
\end{equation}
Now any two elements $a,b$ of a line $L_{h}\in\mathcal{L}_{4}$ satisfy $a\cdot
b=1$ if $a\neq b$ and $a\cdot b=0$ if $a=b.$ So the three points of the line
\begin{equation*}
L_{\mathcal{R}}:=\{U_{\emptyset\emptyset00},U_{\emptyset\emptyset
11},U_{\emptyset\emptyset22}\}\subset\langle L_{c},L_{d}\rangle
\label{L_R = .. 00,11,22}%
\end{equation*}
are perpendicular to every point of $\mathcal{R}.$ However this is not the case
for any other point of $\operatorname*{PG}(7,2),$ and so $\langle
\mathcal{R}\rangle$ is the $5$-flat $(L_{\mathcal{R}})^{\perp}.$ A plane
$P=\mathbb{P}V_{3}$ in the second six, say with equation $\xi_{3}=2\xi
_{4}=-\xi_{4},$ can be treated similarly, with the subset $\theta_{u}(V_{3})$
consisting of the 27 points
\begin{equation}
\mathcal{R}_{\ast}:=\{U_{ijk\overline{k}}|~i,j,k\in\mathbb{F}_{3}=\{0,1,2\}\},
\label{R = Uijkkbar}%
\end{equation}
where, for $k\in\mathbb{F}_{3},$ $\overline{k}$ denotes $-k(=2k).$ If
$\mathcal{R}_{\ast}$ is as in (\ref{R = Uijkkbar}) then we see that
$\langle\mathcal{R}_{\ast}\rangle$ is the $5$-flat $(L_{\mathcal{R}_{\ast}%
})^{\perp}$ where
\begin{equation*}
L_{\mathcal{R}_{\ast}}:=\{U_{\emptyset\emptyset00},U_{\emptyset\emptyset
12},U_{\emptyset\emptyset21}\}\subset\langle L_{c},L_{d}\rangle.
\label{L_R = ... 00,12,21}\qedhere%
\end{equation*}
\end{proof}

\begin{remark}
Consider the triplet $\{\mathcal{R},\mathcal{R}^{\prime},\mathcal{R}%
^{\prime\prime}\}$ of kind $\mathcal{C}_{2}$ which contains $\mathcal{R}$ as
in (\ref{R = Uijkk}). Then $\mathcal{R}^{\prime}=A_{\mu}(\mathcal{R})$ and
$\mathcal{R}^{\prime\prime}=A_{2\mu}(\mathcal{R})$ for suitable $\mu
\in(\mathbb{F}_{3})^{4},$ for example $\mu=0001.$ Consequently
\begin{equation*}
L_{\mathcal{R}^{\prime}}:=\{U_{\emptyset\emptyset01},U_{\emptyset\emptyset
12},U_{\emptyset\emptyset20}\},\quad L_{\mathcal{R}^{\prime\prime}%
}:=\{U_{\emptyset\emptyset02},U_{\emptyset\emptyset10},U_{\emptyset
\emptyset21}\}. \label{L_R' R''}%
\end{equation*}
Similarly, in the case of the triplet $\{\mathcal{R}_{\ast},\mathcal{R}_{\ast
}^{\prime},\mathcal{R}_{\ast}^{\prime\prime}\}$ of kind $\mathcal{C}_{2}$
which contains $\mathcal{R}_{\ast}$ as in (\ref{R = Uijkkbar}) we see that
\begin{equation*}
L_{\mathcal{R}_{\ast}^{\prime}}:=\{U_{\emptyset\emptyset01},U_{\emptyset
\emptyset10},U_{\emptyset\emptyset22}\},\quad L_{\mathcal{R}_{\ast}%
^{\prime\prime}}:=\{U_{\emptyset\emptyset02},U_{\emptyset\emptyset
11},U_{\emptyset\emptyset20}\}. \label{L_R*' R*''}%
\end{equation*}
So the three lines $\{L_{\mathcal{R}},L_{\mathcal{R}^{\prime}},L_{\mathcal{R}%
^{\prime\prime}}\}$ are a regulus in the $3$-flat $\langle L_{c},L_{d}\rangle,$
and the three lines $\{L_{\mathcal{R}_{\ast}},L_{\mathcal{R}_{\ast}^{\prime}%
},L_{\mathcal{R}_{\ast}^{\prime\prime}}\}$ are the opposite regulus, the
$9$-set supporting the two reguli being that hyperbolic quadric
$\mathcal{H}_{3}$ in the $3$-flat $\langle L_{c},L_{d}\rangle$ which has
$L_{c}$ and $L_{d}$ as its two external lines. Of course similar considerations
apply to all twelve of the planes in $\operatorname*{PG}(3,3)$ of kind
$\mathcal{P}_{2}.$ Thus each of the six pairs of lines in $\mathcal{L}_{4}$
gives rise to a pair of opposite reguli and hence to a set of
$6\times2\times3=36$ lines $L\subset\omega_{2}.$ Each such $L$ gives rise to a
rogue $\mathcal{R}$ of kind $\mathcal{C}_{2},$ namely to
$\mathcal{R}=L^{\perp}\cap\omega_{4}.$
\end{remark}

\begin{theorem}
Suppose that a $27$-set $\mathcal{R\subset\omega}_{4}$ is a rogue of kind
$\mathcal{C}_{3}.$ Then $\langle\mathcal{R}\rangle$ is a $6$-flat in
$\operatorname*{PG}(7,2).$
\end{theorem}

\begin{proof}
By Theorem \ref{Thm 40 subgps} there are four triplets in $\omega_{4}$ of kind
$\mathcal{C}_{3},$ which arise from the four planes $\xi_{r}=0,$
$r\in\{1,2,3,4\}.$ Consider the plane $P=\mathbb{P}V_{3}$ with equation
$\xi_{4}=0.$ Then, see (\ref{255 points}), it gives rise to the following
triplet of $27$-sets%
\begin{equation*}
\mathcal{R}:=\{U_{ijk0}\},~\mathcal{R}^{\prime}:=\{U_{ijk1}\},~\mathcal{R}%
^{\prime\prime}:=\{U_{ijk2}\},~~~i,j,k\in\mathbb{F}_{3}=\{0,1,2\}.
\end{equation*}
It quickly follows that $\langle\mathcal{R}\rangle,~\langle\mathcal{R}%
^{\prime}\rangle$ and $\langle\mathcal{R}^{\prime\prime}\rangle$ are $6$-flats,
namely%
\begin{equation*}
\langle\mathcal{R}\rangle=\langle U_{\emptyset\emptyset\emptyset0}\rangle^{\perp
},~\langle\mathcal{R}^{\prime}\rangle=\langle U_{\emptyset\emptyset\emptyset
1}\rangle^{\perp},~\langle\mathcal{R}^{\prime\prime}\rangle=\langle U_{\emptyset
\emptyset\emptyset2}\rangle^{\perp},
\end{equation*}
where $\{U_{\emptyset\emptyset\emptyset0},U_{\emptyset\emptyset\emptyset
1},U_{\emptyset\emptyset\emptyset2}\}=L_{d}.$ Of course the other three
triplets in $\omega_{4}$ of kind $\mathcal{C}_{3}$ are associated in a similar
way with the other three lines $L_{a},L_{b},L_{c}\in\mathcal{L}_{4}.$
\end{proof}

\section{Intersection properties}

\subsection{Introduction}

If $\Delta_{1}$ and $\Delta_{2}$ are any two distinct triplets of $27$-set
denizens of $\omega_{4}$ note that
\begin{equation}
\mathcal{N}(\Delta_{1},\Delta_{2}):=\{R_{1}\cap R_{2}%
:R_{1}\in\Delta_{1},R_{2}\in\Delta_{2}\}\label{ennead N}%
\end{equation}
\emph{is an ennead of }$9$\emph{-sets which provides a partition of }%
$\omega_{4}.$ For suppose that $H_{1}$ and $H_{2}$ are two $(Z_{3})^{3}$
subgroups of $\mathcal{G}_{81}$ whose orbits in $\omega_{4}$ yield the
triplets $\Delta_{1}$ and $\Delta_{2}.$ Then the $(Z_{3})^{2}$ subgroup
$H=H_{1}\cap H_{2}\ $yields one member $\{hu,h\in H\}$ of the ennead of \emph{
}$9$-sets (\ref{ennead N}), the other members of the ennead being the other
orbits of $H$ in $\omega_{4}.$

Since the origin of the present research arose from our interest in Segre
varieties $\mathcal{S}_{3}(2)$ in $\operatorname*{PG}(7,2),$ let us at least
look at the different kinds of intersection\emph{ }$\mathcal{S\cap R}$ of a
Segre variety $\mathcal{S}\subset\omega_{4}$ with another $27$-set denizen
$\mathcal{R}$ of $\omega_{4}.$ Such an intersection we will term a
\emph{section }of the Segre $\mathcal{S}.$ Recall from Theorem
\ref{Thm Only P_0 gives Segres} that a Segre variety $\mathcal{S}\subset
\omega_{4}$ arises as $\mathcal{S}_{H}:=\{hu,h\in H\}$ from a $(Z_{3})^{3}$
subgroup $H<\mathcal{G}_{81}$ which is of class $\mathcal{C}_{0},$ being the
image, under the isomorphism $A$ in (\ref{isomorphism A}), of a $3$%
-dimensional subspace $V_{3}$ such that the projective plane $P=\mathbb{P}%
V_{3}\subset\operatorname*{PG}(3,3)$ is of kind $\mathcal{P}_{0}.$

\begin{lemma}
\label{Lem 13 lines in P of kind P_0}Suppose that $P=\mathbb{P}V_{3}%
\subset\mathbb{P}V(4,3)$ is a projective plane of kind $\mathcal{P}_{0}.$ Then
the $13$ lines $L\subset P$ fall into three $\mathcal{G}(\mathcal{L}_{4}%
)$-orbits:
\begin{itemize}
  \item[(i)] $3$ lines of kind $\Lambda_{4}$;
  \item[(ii)] $6$ lines of kind $\Lambda_{6}$;
  \item[(iii)] $4$ lines of kind $\Lambda_{3}$.
\end{itemize}
\end{lemma}

\begin{proof}
Without loss of generality we may, see Remark
\ref{Rmk all Xi or all Xi* for kind P_0}, consider the particular plane
\begin{equation*}
P=\mathbb{P}V_{\beta\gamma\delta},\text{ \quad where }V_{\beta\gamma\delta
}={\prec}\beta,\gamma,\delta{\succ}\subset V(4,3)=(\mathbb{F}_{3})^{4}.
\label{V_beta gamma delta}%
\end{equation*}
Observe that the $13$ lines in the plane $P$ are as follows:\newline(i) the 3
lines $\mathbb{P}{\prec}\beta,\gamma{\succ},\mathbb{P}{\prec}\beta,\delta
{\succ},\mathbb{P}{\prec}\gamma,\delta{\succ}\ $of weight pattern $(0,2,0,2);$%
\newline(ii) the 6 lines $\mathbb{P}{\prec}\beta,\gamma\pm\delta{\succ}
,\ \mathbb{P}{\prec}\gamma,\beta\pm\delta{\succ},\
\mathbb{P}{\prec}\delta,\beta \pm\gamma{\succ}\ $of weight pattern
$(0,1,2,1);$\newline(iii) the 4 lines
$\mathbb{P}{\prec}\beta\pm\gamma,\beta\pm\delta{\succ}\ $of weight pattern
$(0,3,1,0).$\newline Hence, see (\ref{Table 130 lines}), the stated result
holds.
\end{proof}

\smallskip

\noindent Equivalently expressed, the thirteen $Z_{3}\times Z_{3}$ subgroups of
$\langle A_{\beta},A_{\gamma},A_{\delta}\rangle<\mathcal{G}_{81}$ comprise: (i)
three of class $\mathcal{K}_{4}$ (ii) six of class $\mathcal{K}_{6}$ (iii) four
of class $\mathcal{K}_{3}.$

\subsection{Sections of a Segre variety $\mathcal{S}\subset\omega_{4}$}

Without loss of generality we may consider the particular Segre variety
$\mathcal{S}_{\beta\gamma\delta}:=\mathcal{S}_{H}$ where $H=\langle A_{\beta
},A_{\gamma},A_{\delta}\rangle$:%
\begin{equation*}
\mathcal{S}_{\beta\gamma\delta}=\theta_{u}(V_{\beta\gamma\delta}),\text{ \quad
where }V_{\beta\gamma\delta}={\prec}\beta,\gamma,\delta{\succ}\subset
(\mathbb{F}_{3})^{4}. \label{S_beta gamma delta}%
\end{equation*}
In detail the $27$ elements of $V_{\beta\gamma\delta}$ are:%
\begin{equation}%
\begin{tabular}
[c]{|c@{~~}c@{~~}c|}\hline
$0000$ & $1221$ & $2112$\\
$2121$ & $0012$ & $1200$\\
$1212$ & $2100$ & $0021$\\\hline
\end{tabular}
\ \ ~\underrightarrow{T_{\delta}}~ \ \
\begin{tabular}
[c]{|c@{~~}c@{~~}c|}\hline
$2211$ & $0102$ & $1020$\\
$1002$ & $2220$ & $0111$\\
$0120$ & $1011$ & $2202$\\\hline
\end{tabular}
\ \  ~\underrightarrow{T_{\delta}}~ \ \
\begin{tabular}
[c]{|c@{~~}c@{~~}c|}\hline
$1122$ & $2010$ & $0201$\\
$0210$ & $1101$ & $2022$\\
$2001$ & $0222$ & $1110$\\\hline
\end{tabular}
\ \  . \label{GF(3) 9+9+9}%
\end{equation}
In the display (\ref{GF(3) 9+9+9}) the rows of each $9$-set are orbits of
$\langle T_{\beta}\rangle$ and the columns of each $9$-set are orbits of
$\langle T_{\gamma}\rangle,$ where $T_{\lambda}$ denotes the translation which
maps $\mu\in(\mathbb{F}_{3})^{4}$ to $\lambda+\mu\in(\mathbb{F}_{3})^{4};$ of
course $(T_{\lambda})^{2}=T_{2\lambda}.$ Incidentally observe from
(\ref{GF(3) 9+9+9}) that, in conformity with (\ref{8 type P_0 planes}),
$V_{\beta\gamma\delta}$ is that $3$-dimensional subspace of $V(4,3)$ having
the equation
\begin{equation*}
\xi_{1}+\xi_{2}+\xi_{3}+\xi_{4}=0. \label{xi + xi +xi + xi = 0}%
\end{equation*}
Using the basis $\mathcal{B},$ with $\zeta_{h}$ as in (\ref{zeta-a, ...
,zeta_d}), we see, in the shorthand notation of Eq.~(\ref{8 lines through u}),
that $\mathcal{S}_{\beta\gamma\delta}$ thus consists of the following $27$
points:%
\begin{equation}%
\begin{tabular}
[c]{|c@{~~}c@{~~}c|}\hline 
$u$ & $1256$ & $3478$\\
$5678$ & $56u$ & $78u$\\
$1234$ & $12u$ & $34u$\\\hline
\end{tabular}
\ \  ~\underrightarrow{A_{\delta}}~ \ \
\begin{tabular}
[c]{|c@{~~}c@{~~}c|}\hline 
$2358$ & $25u$ & $38u$\\
$58u$ & $137u$ & $246u$\\
$23u$ & $468u$ & $157u$\\\hline
\end{tabular}
\ \ ~\underrightarrow{A_{\delta}}~ \ \
\begin{tabular}
[c]{|c@{~~}c@{~~}c|}\hline 
$1467$ & $16u$ & $47u$\\
$67u$ & $248u$ & $135u$\\
$14u$ & $357u$ & $268u$\\\hline
\end{tabular}
\ \ . \label{9+9+9}%
\end{equation}
Observe that each $9$-set in (\ref{9+9+9}) is a Segre variety $\mathcal{S}%
_{2}(2),$ whose generators are the rows and columns in the display, these
being orbits of, respectively, $\langle A_{\beta}\rangle$ and $\langle
A_{\gamma}\rangle$.

Acting upon $\mathcal{S}_{\beta\gamma\delta}$ with a $Z_{3}$ subgroup of
$\mathcal{G}_{81}$ which is not in $\langle A_{\beta},A_{\gamma},A_{\delta
}\rangle,$ for example with $\langle A_{\alpha}\rangle,$ will produce the
siblings $\mathcal{S}_{\beta\gamma\delta}^{\prime},~\mathcal{S}_{\beta
\gamma\delta}^{\prime\prime}$ of $\mathcal{S}_{\beta\gamma\delta},$ these
siblings being the images under $\theta_{u}$ of the two affine subspaces in
$V(4,3)=(\mathbb{F}_{3})^{4}$ which are translates of $V_{\beta\gamma\delta}.$

Recalling the proof of Lemma \ref{Lem 13 lines in P of kind P_0}, let us make
the following choices of representatives for the three kinds of $2$%
-dimensional subspaces $V_{2}\subset V_{\beta\gamma\delta}$:
\begin{equation*}
\text{(i) }{\prec}\beta,\gamma{\succ},\quad\text{(ii) }{\prec}\delta,\beta
+\gamma{\succ},\quad\text{(iii) }{\prec}\beta-\gamma,\beta-\delta{\succ}.
\label{3 choices i ii iii}%
\end{equation*}
For the choice (i) the nine elements of $V_{2}={\prec}\beta,\gamma{\succ}\ $are
those in the first $9$-set in (\ref{GF(3) 9+9+9}). The corresponding section
$\theta_{u}(V_{2})$ of $\mathcal{S}_{\beta\gamma\delta}$ is the first of the
three $\mathcal{S}_{2}(2)$ varieties in (\ref{9+9+9}).

For the choice (ii) the nine elements of $V_{2}={\prec}\delta,\beta+\gamma
{\succ}\
$satisfy $\xi_{1}=\xi_{2}$ and are those underlined in:%
\begin{equation*}%
\begin{tabular}
[c]{|c@{~~}c@{~~}c|}\hline
\underline{$0000$} & $1221$ & $2112$\\
$2121$ & \underline{$0012$} & $1200$\\
$1212$ & $2100$ & \underline{$0021$}\\\hline
\end{tabular}
\ \  ~\underrightarrow{T_{\delta}}~ \ \
\begin{tabular}
[c]{|c@{~~}c@{~~}c|}\hline
\underline{$2211$} & $0102$ & $1020$\\
$1002$ & \underline{$2220$} & $0111$\\
$0120$ & $1011$ & \underline{$2202$}\\\hline
\end{tabular}
\ \  ~\underrightarrow{T_{\delta}}~ \ \
\begin{tabular}
[c]{|c@{~~}c@{~~}c|}\hline
\underline{$1122$} & $2010$ & $0201$\\
$0210$ & \underline{$1101$} & $2022$\\
$2001$ & $0222$ & \underline{$1110$}\\\hline
\end{tabular}
\ \ .\label{section (ii)}%
\end{equation*}
We will term the resulting section $\theta_{u}(V_{2})$ of the $\mathcal{S}%
_{3}(2)$ a $3$\emph{-generator set:} it consists of three parallel generators
of $\mathcal{S}_{3}(2)$ which meet a `perpendicular' $\mathcal{S}_{2}(2)$ in
three points no two of which lie on the same generator of the $\mathcal{S}%
_{2}(2).$

Finally the nine elements of $V_{2}={\prec}\beta-\gamma,\beta-\delta{\succ}$
satisfy $\xi_{4}=0$ and are those underlined in:%
\begin{equation*}%
\begin{tabular}
[c]{|c@{~~}c@{~~}c|}\hline
\underline{$0000$} & $1221$ & $2112$\\
$2121$ & $0012$ & \underline{$1200$}\\
$1212$ & \underline{$2100$} & $0021$\\\hline
\end{tabular}
\ \ ~\underrightarrow{T_{\delta}}~ \ \
\begin{tabular}
[c]{|c@{~~}c@{~~}c|}\hline
$2211$ & $0102$ & \underline{$1020$}\\
$1002$ & \underline{$2220$} & $0111$\\
\underline{$0120$} & $1011$ & $2202$\\\hline
\end{tabular}
\ \ ~\underrightarrow{T_{\delta}}~ \ \
\begin{tabular}
[c]{|c@{~~}c@{~~}c|}\hline
$1122$ & \underline{$2010$} & $0201$\\
\underline{$0210$} & $1101$ & $2022$\\
$2001$ & $0222$ & \underline{$1110$}\\\hline
\end{tabular}
\ \ . \label{section (iii)}%
\end{equation*}
We will term the resulting section $\theta_{u}(V_{2})$ of the
$\mathcal{S}_{3}(2)$ a \emph{fan:}

\begin{definition}
A subset $\mathcal{F}$ of nine points of a $\mathcal{S}_{3}(2)$ is a
\emph{fan} (= \emph{f}ar-\emph{a}part \emph{n}ine) if no two points of
$\mathcal{F}$ lie on the same generator. (So if $\mathcal{F}$ is a fan for a
$\mathcal{S}_{3}(2)$ then the 3 generators through each of the 9 points of
$\mathcal{F}$ account for all $3\times9=27$ generators of $\mathcal{S}%
_{3}(2).$)
\end{definition}

We may summarize the foregoing as follows.

\begin{theorem}
A section of a Segre variety $\mathcal{S}_{3}(2)$ in $\omega_{4}$ is either
(i) a $\mathcal{S}_{2}(2),$ or (ii) a 3-generator set, or (iii) a fan.
\end{theorem}

\subsection{Hamming distances and troikas}

In the $\operatorname*{GF}(3)$ space $V(4,3)=(\mathbb{F}_{3})^{4}$ we have
been employing the basis $\mathcal{B}_{\varepsilon}:=\{\varepsilon
_{1},\varepsilon_{2},\varepsilon_{3},\varepsilon_{4}\},$ see
(\ref{epsilon_1,2,3,4}). Using this basis it helps at times to make use of the
associated \emph{Hamming distance} $\operatorname*{hd}\nolimits_{\varepsilon
}(\rho,\sigma)$ between two elements $\sigma,\tau\in(\mathbb{F}_{3})^{4},$ as
defined by%
\begin{equation*}
\operatorname*{hd}\nolimits_{\varepsilon}(\rho,\sigma)=\operatorname{wt}%
_{\varepsilon}(\rho-\sigma). \label{GF(3)Hamming distance}%
\end{equation*}

\begin{remark}
If $\rho,\sigma$ are belong to the same row in (\ref{XiXi* alpha ... delta*})
observe that $\operatorname*{hd}_{\varepsilon}(\rho,\sigma)=2,$ while if
$\rho,\sigma$ belong to different rows then $\operatorname*{hd}_{\varepsilon
}(\rho,\sigma)$ is odd.
\end{remark}

The next lemma demonstrates that some aspects of orthogonality in the
$\operatorname*{GF}(2)$ space $\operatorname*{PG}(7,2)$ can be neatly dealt
with in $\operatorname*{GF}(3)$ terms.

\begin{lemma}
\label{Lem p.p}Two points $p_{\rho},p_{\sigma}\in\omega_{4}$ are orthogonal or
non-orthogonal according as $\operatorname*{hd}_{\varepsilon}(\rho,\sigma)$ is
even or odd.
\end{lemma}

\begin{proof}
For two points $u_{h}(i),u_{h}(j)\in L_{h}$ we have $u_{h}(i)\cdot u_{h}%
(j)=1+\delta_{ij}.$ Hence if $\rho=ijkl$ and $\sigma=i^{\prime}j^{\prime
}k^{\prime}l^{\prime}$ it follows that $p_{\rho}\cdot p_{\sigma}=\delta
_{ii^{\prime}}+\delta_{jj^{\prime}}+\delta_{kk^{\prime}}+\delta_{ll^{\prime}%
},$ whence the stated result.
\end{proof}

A description of the different sections of a Segre variety $\mathcal{S}%
\subset\omega_{4}$ can sometimes be helped by the use of the alternative basis
$\mathcal{B}_{\Xi}:=\{\beta,\gamma,\delta,\alpha\}$ for $V(4,3)=(\mathbb{F}%
_{3})^{4}.$ Here, as in (\ref{XiXi* alpha ... delta*}), $\mathcal{\beta
}=1221,~\mathcal{\gamma}=2121,~\mathcal{\delta}=2211,~\mathcal{\alpha}=1111.$
The change of basis equations are therefore:%
\begin{align}
  \beta = {}   &   \varepsilon_{1}  -\varepsilon_{2}  -\varepsilon_{3}  +\varepsilon_{4},&
  \gamma ={}   &  -\varepsilon_{1}  +\varepsilon_{2}  -\varepsilon_{3}  +\varepsilon_{4},\nonumber\\
  \delta ={} &   -\varepsilon_{1}  -\varepsilon_{2} +\varepsilon_{3}  +\varepsilon_{4},&
  \alpha={}&\varepsilon_{1}  +\varepsilon_{2}  +\varepsilon_{3}  +\varepsilon_{4},\nonumber\\
  \varepsilon_{1}={}  &   \beta -\gamma-\delta+\alpha,&
  \varepsilon_{2}={}  & -\beta+\gamma-\delta+\alpha,\nonumber\\
    \varepsilon_{3} ={}&-\beta-\gamma+\delta+\alpha,&
    \varepsilon_{4}  ={}&\beta+\gamma+\delta+\alpha.
\label{change of basis relations}%
\end{align}
Observe that the chosen ordering of the elements of the basis $\mathcal{B}%
_{\Xi}$ results in the change of basis matrix $\mathbf{M}$ having the simple
properties $\mathbf{M}^{\text{t}}=\mathbf{M}=\mathbf{M}^{-1}.$

So in $V(4,3)$ we now have available the Hamming distance in the basis
$\mathcal{B}_{\Xi},$ namely
\begin{equation*}
\operatorname*{hd}\nolimits_{\Xi}(\rho,\sigma):=\operatorname{wt}_{\Xi}%
(\rho-\sigma),\quad\rho,\sigma\in(\mathbb{F}_{3})^{4},
\label{Hamming distance using Xi}%
\end{equation*}
where $\operatorname{wt}_{\Xi}(\lambda)$ denotes the weight of an element
$\lambda\in V(4,3)$ in the basis $\mathcal{B}_{\Xi}.$

\begin{lemma}
\label{Lem Two bases Epsilon and Xi}%
\begin{align}
\operatorname{wt}_{\Xi}(\lambda)  &  =2\Longleftrightarrow\operatorname{wt}%
_{\varepsilon}(\lambda)=2,~~~\operatorname{wt}_{\Xi}(\lambda
)=3\Longleftrightarrow\operatorname{wt}_{\varepsilon}(\lambda)=3,\nonumber\\
\operatorname{wt}_{\Xi}(\lambda)  &  =1\Longleftrightarrow\operatorname{wt}%
_{\varepsilon}(\lambda)=4,~~~\operatorname{wt}_{\Xi}(\lambda
)=4\Longleftrightarrow\operatorname{wt}_{\varepsilon}(\lambda)=1.
\label{Two weight results}%
\end{align}

\end{lemma}

\begin{proof}
The results (\ref{Two weight results}) follow immediately from
(\ref{change of basis relations}).
\end{proof}


Let us also define the Hamming distance $\operatorname*{hd}(p_{\rho}%
,p_{\sigma})$ between two points $p_{\rho},p_{\sigma}\in\omega_{4}$ to be%
\begin{equation*}
\operatorname*{hd}(p_{\rho},p_{\sigma}):=\operatorname*{hd}\nolimits_{\Xi
}(\rho,\sigma). \label{hd for points}%
\end{equation*}
Observe that this definition does not depend upon the choice of the point $u$
used in the bijective correspondence $\theta_{u}$ in (\ref{Def theta_u}), due
to the invariance of $\operatorname*{hd}\nolimits_{\Xi}$ under translations:
$\operatorname*{hd}\nolimits_{\Xi}(\lambda+\rho,\lambda+\sigma
)=\operatorname*{hd}_{\Xi}(\rho,\sigma).$ Suppose that we confine our attention
to points $p,p^{\prime},\ldots$ on a particular Segre variety in $\omega_{4},$
say $\mathcal{S=S}_{\beta\gamma\delta}$ in (\ref{9+9+9}). Then observe that
$\operatorname*{hd}(p,p^{\prime})=d,~d\in\{1,2,3\},$ provided that $p^{\prime}$
can be obtained from $p$ only by the use of at least $d$ of
the generating subgroups $\langle A_{\beta}\rangle,\langle A_{\gamma}%
\rangle,\langle A_{\delta}\rangle$ of $\mathcal{S}.$ In particular distinct
points $p$ and $p^{\prime}$ lie on the same generator of $\mathcal{S}$ if and
only if $\operatorname*{hd}(p,p^{\prime})=1.$ This Hamming distance on
$\mathcal{S}$ appears also in \cite{HavlicekOdehnalSaniga}.

\begin{definition}
(i) A \emph{troika }on a given $\mathcal{S}_{3}(2)$ variety $\mathcal{S}$ in
$\omega_{4}$ is a set of three points of $\mathcal{S}$ which are
$\operatorname*{hd}\nolimits_{\Xi}=3$ apart.

(ii) The \emph{centre} of a troika $t=\{p_{1},p_{2},p_{3}\}$ is the point
$c(t)=p_{1}+p_{2}+p_{3}.$
\end{definition}

For example, the three points $u,246u,357u\in\mathcal{S}_{\beta\gamma\delta}$
in (\ref{9+9+9}) are $\operatorname*{hd}\nolimits_{\Xi}=3$ apart and so form a
troika. Alternatively this follows from (\ref{Two weight results}) since
$0000,0111,0222$ in (\ref{GF(3) 9+9+9}) are $\operatorname*{hd}_{\varepsilon
}=3$ apart.

\begin{theorem}
\label{Thm fans, troikas, and L_4}(i) A fan $\mathcal{F}$ for a Segre variety
$\mathcal{S}$ in $\omega_{4}$ can be uniquely expressed as the union
$\mathcal{F}=t\cup t^{\prime}\cup t^{\prime\prime}$ of a triplet of troikas.

(ii) The three troikas $t,t^{\prime},t^{\prime\prime}$ in $\mathcal{F}$ share
the same centre, say $c_{\mathcal{F}}.$

(iii) Moreover a fan $\mathcal{F}$ for $\mathcal{S}$ determines uniquely a
triplet $\mathcal{T}=\{\mathcal{F},\mathcal{F}^{\prime},\mathcal{F}%
^{\prime\prime}\}$ of fans such that $\mathcal{F}\cup\mathcal{F}^{\prime}%
\cup\mathcal{F}^{\prime\prime}=\mathcal{S}.$ Further, $L(\mathcal{T}%
):=\{c_{\mathcal{F}},c_{\mathcal{F}^{\prime}},c_{\mathcal{F}^{\prime\prime}%
}\}$ is one of the lines of the $\mathcal{H}_{7}$-tetrad $\mathcal{L}_{4}.$
\end{theorem}

\begin{proof}
Suppose that $\mathcal{S}=\theta_{u}(V_{3})$ and that $\mathcal{F}%
\subset\mathcal{S}$ is a fan which contains $u.$

(i) So $\mathcal{F}=\theta_{u}(V_{2}),$ $V_{2}\subset V_{3},$ where the
projective line $L=\mathbb{P}V_{2}$ is of kind $\Lambda_{3},$ with line pattern
$(0,3,1,0).$ Consequently $L$ has a unique point $\langle\lambda\rangle$ with
$\operatorname{wt}(\lambda)=3.$ Hence the element $0000\in(\mathbb{F}_{3}%
)^{4}$ has the unique extension $V_{1}:=\{0000,\lambda,2\lambda\}$ to a
$3$-set of elements of $V_{2}$ which are Hamming distance $3$ apart. So the
point $u$ belongs to a unique troika $t\subset\mathcal{F},$ namely
$t:=\theta_{u}(V_{1})=\{u,A_{\lambda}u,A_{2\lambda}u\}.$ If $V_{2}%
={\prec}\lambda,\mu{\succ},$ the two translates
$T_{\mu}(V_{1}),T_{2\mu}(V_{1})$
of $V_{1}$ in $V_{2}$ yield two other troikas $t^{\prime}=A_{\mu}%
t,t^{\prime\prime}=A_{2\mu}t,$ giving rise to the claimed unique decomposition
$\mathcal{F}=t\cup t^{\prime}\cup t^{\prime\prime}.$

(ii) Since $\operatorname{wt}(\lambda)=3,$ precisely one of the coordinates of
$\lambda$ in the basis $\mathcal{B}_{\varepsilon}$ is zero. First suppose
$\lambda$ satisfies $\xi_{4}=0.$ Then for $\mathbb{P}V_{2}=\mathbb{P}%
({\prec}\lambda,\mu{\succ})$ to be of kind $\Lambda_{3}$ the element $\mu$ must
also satisfy $\xi_{4}=0$. So a point $p\in\mathcal{F}$ must be of the form
$U_{ijk0}$. Hence that troika $\{p,A_{\lambda}p,A_{2\lambda
}p\}\subset\mathcal{F}$ which contains $p$ has centre%
\begin{equation}
c=(I+A_{\lambda}+(A_{\lambda})^{2})U_{ijk0}=U_{\emptyset\emptyset\emptyset0}.
\label{common centre c if xi_4 = 0}%
\end{equation}
So the same point $c=U_{\emptyset\emptyset\emptyset0},$ which lies on the line
$L_{d}\in\mathcal{L}_{4},$ is the centre of each of the three troikas in
$\mathcal{F}.$ Of course, if instead $\lambda$ satisfies $\xi_{i}=0$ for
$i=1,2,3$ then the analogous reasoning shows that the common centre of the
three troikas in $\mathcal{F}$ is $U_{0\emptyset\emptyset\emptyset}\in
L_{a},~U_{\emptyset0\emptyset\emptyset}\in L_{b}$ or $U_{\emptyset
\emptyset0\emptyset}\in L_{c},$ according as $i=1,2$ or $3.$

(iii) We are dealing with $\mathcal{S}=\theta_{u}(V_{3})$ where $V_{3}$ is of
the form $V_{3}=V_{2}
\oplus{{\prec}\nu{\succ}},~V_{2}={{\prec}\lambda,\mu{\succ}},$ and where may
choose $\nu$ to have $\operatorname{wt}_{\varepsilon}=4.$ The fan
$\mathcal{F}=\theta_{u}(V_{2})$ determines a triplet $\mathcal{T}%
=\{\mathcal{F},\mathcal{F}^{\prime},\mathcal{F}^{\prime\prime}\}$ of fans,
where $\mathcal{F}^{\prime}=A_{\nu}(\mathcal{F})$ and $\mathcal{F}^{\prime
\prime}=A_{2\nu}(\mathcal{F}),$ such that $\mathcal{F}\cup\mathcal{F}^{\prime
}\cup\mathcal{F}^{\prime\prime}=\mathcal{S}.$ Moreover if $c_{\mathcal{F}%
}=U_{\emptyset\emptyset\emptyset0}$ as in (\ref{common centre c if xi_4 = 0})
then $c_{\mathcal{F}^{\prime}}=A_{\nu}c_{\mathcal{F}},$~ $c_{\mathcal{F}%
^{\prime\prime}}=A_{2\nu}c_{\mathcal{F}}$ will be the other two points
$U_{\emptyset\emptyset\emptyset1},$~$U_{\emptyset\emptyset\emptyset2}$ of the
line $L_{d}.$ Similarly for the other three cases considered in (ii) above,
where $\{c_{\mathcal{F}},c_{\mathcal{F}^{\prime}},c_{\mathcal{F}^{\prime
\prime}}\}$ is one of the other lines of the tetrad $\mathcal{L}_{4}.$
\end{proof}

In the paper \cite{AspectsSegreinDCC} the fact that a Segre variety
$\mathcal{S}=\mathcal{S}_{3}(2)$ in $\operatorname*{PG}(7,2)$ determines a
distinguished tetrad $\mathcal{L}_{4}$ of lines which span $\operatorname*{PG}%
(7,2)$ only emerged rather late, see \cite[Section 4.1]{AspectsSegreinDCC}.
One of the motivations for the present paper was to come to a clearer
understanding of the relationship between $\mathcal{S}$ and $\mathcal{L}_{4}.$
This can now be achieved: see the next theorem, where it is shown how to
obtain the same tetrad $\mathcal{L}_{4}$ from any of the $24$ copies of a
Segre variety $S_{3}(2)$ in the $81$-set $\omega_{4}.$

\begin{theorem}
A Segre variety $\mathcal{S}$ in $\omega_{4}$ determines precisely four
triplets $\mathcal{T}_{i},~i\in(1,2,3,4)$ of fans; further the resulting four
lines $L_{i}:=L(\mathcal{T}_{i})$ are the four lines of the $\mathcal{H}_{7}%
$-tetrad $\mathcal{L}_{4}$.
\end{theorem}

\begin{proof}
A Segre variety in $\omega_{4}$ is of the form $\mathcal{S}=\theta_{u}(V_{3})$
where the projective plane $P=\mathbb{P}V_{3}\subset\operatorname*{PG}(3,3)$
is of kind $\mathcal{P}_{0}.$ Now, see (\ref{GF(3) 9+9+9}), there are
precisely $8$ elements of $V_{3}$ of weight $3,$ and these form $4$ pairs, say
$\{\pm\lambda_{1}\},$\ $\{\pm\lambda_{2}\},$\ $\{\pm\lambda_{3}\},$%
\ $\{\pm\lambda_{3}\}.$ Consequently the element $0000\in V_{3}$ has precisely
$4$ extensions, namely $\{0000,\lambda_{i},-\lambda_{i}\},\ i=1,2,3,4,$ to a
$3$-set of elements of $V_{3}$ which are Hamming distance $3$ apart. Hence the
point $u=\theta_{u}(0000)$\ lies in precisely $4$ troikas and, by Theorem
\ref{Thm fans, troikas, and L_4}(i), in precisely $4$ fans $\mathcal{F}%
_{i},\ i=1,2,3,4.$ By Theorem \ref{Thm fans, troikas, and L_4}(iii) each of
the resulting $4$ triplets $\mathcal{T}_{i}=\{\mathcal{F}_{i},\mathcal{F}%
_{i}^{\prime},\mathcal{F}_{i}^{\prime\prime}\}$ of fans determines a line
$L(\mathcal{T}_{i})$ of the tetrad $\mathcal{L}_{4}.$ Further, from the proof
of Theorem \ref{Thm fans, troikas, and L_4}(iii), we see that these lines are distinct.
\end{proof}

\subsection{Future research}

We are of the opinion that some further investigation of the denizens of
$\omega_{4},$ and of their interactions, should prove worthwhile. Moreover such
an investigation should not be confined to the $27$-set denizens arising from
Theorem \ref{Thm 40 subgps}, since at least some of the $9$-set denizens of
$\omega_{4}$ arising from Theorem \ref{Thm 130 subgps} deserve attention. In
particular the $9$-sets in $\omega_{4}$ which arise from those lines in the
table (\ref{Table 130 lines}) which have weight pattern $(0,0,4,0)$ are
certainly noteworthy. For suppose that $V(2,3)\subset V(4,3)$ is such that
$L=\mathbb{P}V(2,3)$ is of kind $\Lambda_{7},$ and so $\operatorname{wt}%
_{\varepsilon}(\rho)=3$ for every nonzero element $\rho\in V(2,3).$ It follows
that any pair of distinct elements $\rho,\sigma$ of $V(2,3)$ are Hamming
distance $3$ apart, and hence, by Lemma \ref{Lem p.p}, the points $p_{\rho
},p_{\sigma}$ are non-perpendicular: $p_{\rho}\cdot p_{\sigma}=1.$ \emph{In
this easy manner we have constructed a }$9$\emph{-cap }$\mathcal{N}=\{p_{\rho
}\}_{\rho\in V(2,3)}$ \emph{on the quadric} $\mathcal{H}_{7}.$ Moreover under
the action on $\mathcal{N}$ of $\mathcal{G}_{81}$ we will obtain a partition of
the $81$-set $\omega_{4}$ into an ennead of quadric $9$-caps.

\bigskip

{\small\noindent Ron \thinspace Shaw, Centre for Mathematics,}

{\small \noindent University of Hull, Hull HU6 7RX, UK}

{\small r.shaw@hull.ac.uk}

{\small \medskip}

{\small \noindent Neil Gordon, Department of Computer Science, }

{\small \noindent University of Hull, Hull HU6 7RX, UK}

{\small n.a.gordon@hull.ac.uk}

{\small \medskip}

{\small \noindent Hans Havlicek, Institut f\"{u}r Diskrete Mathematik und
Geometrie,}

{\small \noindent Technische Universit\"{a}t, Wiedner Hauptstra{\ss}e 8--10/104, A-1040
Wien, Austria}

{\small havlicek@geometrie.tuwien.ac.at}

\end{document}